\documentclass[11pt]{amsart}
\usepackage[utf8x]{inputenc}
\usepackage{amsfonts}
\usepackage{amssymb,epsfig}
\usepackage{amsmath}
\usepackage{delarray}
\usepackage{graphicx}
\usepackage{color}
\newtheorem{thm}{Theorem}[section]

\newtheorem{cor}[thm]{Corollary}

\newtheorem{lemma}[thm]{Lemma}

\newtheorem{prop}[thm]{Proposition}

\newtheorem{defn}[thm]{Definition}
\theoremstyle{remark}

\numberwithin{equation}{section}
\newtheorem{example}[thm]{Example}

\newcommand{\fg}{\mathfrak g}

\newcommand{\cF}{\mathcal F}

\newcommand{\cL}{\mathcal L}

\newcommand{\cX}{\mathcal X}

\newcommand{\bbT}{\mathbb T}

\newcommand{\bbS}{\mathbb S}
\newcommand{\bbZ}{\mathbb Z}

\newcommand{\CIT}{\mathrm{CIT}}

\begin{document}

\title{Commuting  foliations}

\author{Nguyen Tien Zung,\;  Truong Hong Minh}
\address{Institut de Mathématiques de Toulouse, UMR5219, Université Toulouse 3}
\email{tienzung.nguyen@math.univ-toulouse.fr}
\email{hong-minh.truong@math.univ-toulouse.fr}

\begin{abstract}{%
The aim of this paper is to extend the notion of commutativity of vector fields to the category of singular 
foliations, using Nambu structures, i.e. integrable multi-vector fields. We will classify the relationship between singular foliations and Nambu structures, and show  some basic results about
commuting Nambu structures.
}\end{abstract}

\date{Version 2, October  2013}

\maketitle

\section{Introduction}

Foliations are often viewed as generalizations of dynamical systems. A lot
of natural notions can be extended from  the world of dynamical systems to the world of foliations.
For example, many authors studied the entropy of foliations  (see, e.g.,  \cite{Walczak-Foliations}
and references therein, and also \cite{Zung-Entropy2011}). However, to our knowledge, 
the general notions of commutativity and integrability, which are  very important for dynamical systems, have 
been given very little attention so far in the world of foliations. The only explicit mentions of the words
``(non)commuting foliations'' that we found in the literature are a paper of
Movshev \cite{Movshev-YM2009} and some papers of Katok and his collaborators  
(see, e.g., \cite{EK-NonCommuting2005}). The aim of this paper is to attract attention to commuting
foliations, and to give some basic results about them, which are relatively simple but nevertheless
interesting in our view. We will also clarify the (somewhat confusing until now) relationship between singular foliations and Nambu structures. 

So what are commuting foliations? 
As V.I. Arnold said, a good definition is five good examples. 
Instead of giving a formal definition, let us list here some examples which could be considered
as (singular) commuting foliations. These examples are different but related to each other, 
and they all generalize the notion of commuting vector fields:

- Commuting actions of Lie groups and Lie algebras. In particular, foliations generated by commuting
vector fields in an integrable dynamical system. 

- A pair of transverse foliations $\cF_1,\cF_2$ such that the holonomy of $\cF_1$ acts trivially
on $\cF_2$ and vice versa (Movshev \cite{Movshev-YM2009}).

- Almost direct products, i.e. constructions of the type $(F_1 \times F_2)/G,$ where $F_1$ and $F_2$ are two
manifolds, and $G$ is a discrete group which acts on the product $F_1 \times F_2$  freely and diagonally,
i.e. $G$ acts on both $F_1$ and $F_2$, and the action of $G$ on 
$F_1 \times F_2$ is composed of these two actions of $G$ on $F_1$ and $F_2$. 
The two foliations on $(F_1 \times F_2) / G$
are the ``horizontal'' foliation with leaves $(F_1 \times \{pt\})/G$ and the ``vertical'' foliation with leaves $(\{pt\} \times F_2)/G$
respectively.

- Foliations generated by compatible Poisson structures, compatible symplectic structures, compatible Dirac
structures, etc.

- Parallelizable webs (see, e.g., \cite{AG-Webs2000}).

- Commuting Nambu structures (see below).

There are many different ways to generate (singular) foliations. One of the most general and convenient 
ways is via the so called Nambu structures,  i.e. multi-vector fields which are integrable à la Frobenius.
In this paper we will also be mainly concerned with Nambu structures.
We refer the reader to Chapter 6 of \cite{DufourZung-PoissonBook} and also to \cite{Zung-Nambu2012} 
for basic notions about Nambu structures which will be used in this paper. 
We will use the following definition of Nambu structures \cite{Zung-Nambu2012} 
which is a bit different from the original definition of Takhtajan  \cite{Takhtajan-Nambu1994}:
A {\bf Nambu structure} (or tensor) of order $q$ on a manifold $M$ is a $q$-vector field $\Lambda$ on $M$ which satisfies
the following  condition: for any point $p \in M$ such that $\Lambda(p) \neq 0$, there is
a local coordinate system $(x_1,\hdots,x_n)$ in a neighborhood of $p$ such that 
\begin{equation}
 \Lambda= \frac{\partial}{\partial x_1} \wedge \hdots \wedge \frac{\partial}{\partial x_q}
\end{equation}
in that neighborhood. Geometrically, a Nambu structure of order $q$ is nothing but a  singular $q$-dimensional foliation together with a contravariant volume form
(i.e. $q$-vector field) on its leaves: in the local normal form 
$\Lambda=\frac{\partial}{\partial x_1} \wedge \hdots \wedge \frac{\partial}{\partial x_q}$ the foliation is generated
by the commuting vector fields $\frac{\partial}{\partial x_1}, \hdots , \frac{\partial}{\partial x_q}.$ A Nambu structure is called \emph{regular}
if it does not vanish anywhere. In that case its foliation is also a regular foliation.

We will clarify the relationship between Nambu structures and foliations in Section 2, before studying commuting foliation (in term of commuting Nambu structures) in the rest of the paper.   

An advantage of multi-vector fields in the study of foliations is that we can use the Schouten bracket in the calculus of multi-vector fields. 
In particular, it would be natural to say that when two Nambu stuctures commute
then their Schouten bracket vanish. This definition, which was already mentioned in \cite{DufourZung-PoissonBook}, works well
in the case when the sum of the orders of the two Nambu structures does not exceed the dimension of the manifold 
(see Definition \ref{defn:CommutingNambu1} and Proposition \ref{prop:CommutingNambu1} below). However, if $\Lambda_1$ and
$\Lambda_2$ are two Nambu structures of orders $q_1$ and $q_2$ respectively on a $n$-dimensional manifold, such that
$q_1 + q_2 > n,$ then the condition $[\Lambda_1,\Lambda_2] = 0$ does not mean much. (It means nothing at all when 
$q_1 + q_2 \geq n + 2,$ because the order of $[\Lambda_1,\Lambda_2]$ is $q_1 + q_2 - 1 > n$ in that case so $[\Lambda_1,\Lambda_2]$
is automatically zero). In this paper, we will give a meaningful definition of commutativity of two Nambu structures $\Lambda_1$
and $\Lambda_2$ in the case  when $q_1 + q_2 > n,$ by replacing  the equation $[\Lambda_1,\Lambda_2] = 0$ by an appropriate  
stronger condition (see Definition \ref{defn:CommutingNambu2} and Proposition \ref{prop:CommutingNambu4}).

\section{From foliations to Nambu structures and back}

In this section, we will show how to go from a singular foliation to a Nambu structure and back, via the notions of associated Nambu stucture and associated foliation, and what changes and what remains when one goes around a full circle. We want to convince the reader that  Nambu structures are the right objects to consider in order to study singular foliations by algebraic means.

\subsection{Tangent and associated Nambu structures}

Given a singular foliation $\mathcal{F}$ of dimension $q$ on a manifold $M$ and a Nambu structure $\Lambda$ of order $q$ on $M$,
it would be natural to say that  $\Lambda$ is \emph{tangent} to $\mathcal{F}$ if for any point $x$ of $M$ such that $\Lambda(x)\neq 0$ we have
that $x$ is a regular point of $\mathcal{F}$ and locally near $x$ we can write
$$\Lambda=\frac{\partial}{\partial x_1}\wedge\hdots \wedge \frac{\partial}{\partial x_q},$$
and $\frac{\partial}{\partial x_1},\hdots,\frac{\partial}{\partial x_q}$ generate $\mathcal{F}$ near $x$.

If $\Lambda$ is tangent to $\mathcal{F}$ in the above sense, then evidently $S(\mathcal{F})\subset S(\Lambda)$, where $S(\mathcal{F})$, $S(\Lambda)$ are the singular set of $\mathcal{F}$ and $\Lambda$ respectively. A simple way to generate a local non-trivial Nambu structure tangent to a singular foliation $\mathcal{F}$ is to take $q$ local vector fields $X_1,\hdots, X_q$ which are tangent to $\mathcal{F}$ and which are linearly independent almost everywhere, and put
\begin{equation*}
\Pi=X_1\wedge\hdots\wedge X_q,
\end{equation*}
then clearly $\Pi$ is a Nambu structure tangent to $\mathcal{F}$. However, this Nambu structure $\Pi$ can be ``bad'', in the sense that it has too many singular points compared to the singularities of $\mathcal{F}$. Ideally, we would like to have a local Nambu structure which vanishes only at singular points of $\mathcal{F}$.
However, such a Nambu structure may or may not exist,  as the following example shows:

\begin{example}\label{ex:1} Consider a germ of singular foliation $\mathcal{F}$ of dimension $2$ at the origin of $\mathbb{C}^3$ whose leaves are $\{z=c\}$ with $c\neq 0$, $\{z=0,y\neq0\}$ and $\{z=0, y=0\}$. Then  $S(\mathcal{F})=\{z=0, y=0\}$. If $\Lambda$ is a Nambu structure tangent to $\mathcal{F}$ in the above sense, then it must have the form
$$\Lambda=f\cdot\frac{\partial}{\partial x}\wedge \frac{\partial}{\partial y},$$
where $f$ is a function vanishing on $S(\mathcal{F})$. Since $\mathrm{codim}S(\mathcal{F})=2$, there is no function which vanishes only on $S(\mathcal{F})$. Moreover, the singular set $S(\Lambda)=\{f=0\}$ has codimension $1$ which is too big compared to $S(\mathcal{F})$.     
\end{example}

This leads us to the following definition:

\begin{defn}\label{def:assoNambu}
A (local or global) Nambu structure $\Lambda$ of order $q$ is called {\bf tangent} to a singular foliation $\cF$ of dimension $q$ on the manifold $M$ if it satisfies the following condition: 
\begin{itemize}
\item[1.] Locally near every point $x\notin S(\Lambda)\cup S(\mathcal{F})$ of $M$ we can write
$$\Lambda=\frac{\partial}{\partial x_1}\wedge\hdots \wedge \frac{\partial}{\partial x_q},$$
and $\frac{\partial}{\partial x_1},\hdots,\frac{\partial}{\partial x_q}$ generate $\mathcal{F}$ near $x$;
\item[2.] $\mathrm{codim}\big(S(\cF)\setminus S(\Lambda)\big)\geq 2$.
\end{itemize}
We will say that $\Lambda$ is an {\bf associated Nambu structure}   if it is a tangent Nambu structure which satisfies the following additional conditions:
\begin{itemize}
\item[3.] $\mathrm{codim}\big(S(\Lambda)\setminus S(\cF)\big)\geq 2$. 
\item[4.] $S(\Lambda)$ is without multiplicity  in the sense that $\Lambda$ can't be written as 
$$\Lambda=f^2\Lambda',$$
where $f$ is a function which vanishes somewhere ($f(O)=0$ in the local case).
\end{itemize}
\end{defn}
A direct consequence of Conditions 2 and 3 above is  that if $\Lambda$ is an associated Nambu structure of a holomorphic foliation $\cF$ and $\mathrm{codim}S(\Lambda)=1$ or $\mathrm{codim}S(\mathcal{F})=1$ then we must have $S(\Lambda)=S(\mathcal{F})$. In Example \ref{ex:1}, an associated Nambu structure of $\mathcal{F}$ is 
$$\frac{\partial}{\partial x}\wedge \frac{\partial}{\partial y}.$$
The following proposition shows that, at least in the holomorphic category, an associated Nambu structure can be seen as a generator of the module  of tangent Nambu structures. Associated Nambu structures do not necessarily exist globally, but they exist locally, and the sheaf of local tangent holomorphic Nambu structures forms a holomorphic line bundle over the manifold, which is often called in the literature the \emph{anti-canonical line bundle} of the foliation. 

\begin{prop}
If $\mathcal{F}$ is holomorphic (i.e. generated by a family of local holomorphic vector fields) then locally near each point $x\in M$, there is a local associated Nambu structure $\Lambda$ which is unique up to multiplication by a function which does not vanish at $x$. Moreover, if $\Pi$ is another local tangent Nambu structure then
\begin{equation}
\Pi=f\Lambda,
\end{equation}
for some local holomorphic function $f$.
\end{prop}
\begin{proof}
Let us start with an arbitrary non-trivial holomorphic Nambu structure $\Pi$ tangent to $\mathcal{F}$, e.g. $\Pi=X_1\wedge\hdots\wedge X_q$ as above. There are two possibilities: either $\mathrm{codim}S(\Pi)=1$, or $\mathrm{codim}S(\Pi)\geq 2$; where $S(\Pi)$ is the singular set of $\Pi$.

If $\mathrm{codim}S(\Pi)=2$, it is clear that $\Pi$ is an associated Nambu structure.   

If $\mathrm{codim}S(\Pi)=1$, then it is the zero locus of some reduced function $g$, i.e. without multiplicity in its decomposition. Then $\frac{\Pi}{g}$ is holomorphic and so it is again a tangent Nambu structure. If $\mathrm{codim}S(\frac{\Pi}{g})=1$, then we can device $\frac{\Pi}{g}$ again by some other function which vanish at $x$ and so on. At the end, we obtain 
$$\Pi=h\Lambda,$$
where $\mathrm{codim}S(\Lambda)\geq 2.$  The problem which may arise now is that $\mathcal{F}$ itself may have a singular set of codimension $1$. If it is the case, we will find a reduced function $s$ such that the singular set of $\mathcal{F}$ is the zero locus of $s$. Then $s\Lambda$ will be an associated Nambu structure.

Let us now prove the uniqueness. Suppose that $\Lambda'$ is another associated Nambu structure, then it is clear that 
\begin{equation*}
\mathrm{codim}\Delta\big(S(\Lambda),S(\Lambda')\big)\geq 2, 
\end{equation*}
where $\Delta\big(S(\Lambda),S(\Lambda')\big)=\big(S(\Lambda)\setminus S(\Lambda')\big)\cup\big(S(\Lambda')\setminus S(\Lambda)\big)$, and
\begin{equation*}
\Lambda'=f\Lambda,
\end{equation*} 
for some holomorphic function $f$ which is defined outside the union of two singular sets $S(\Lambda)\cup S(\Lambda')$.

If $\mathrm{codim}S(\Lambda)\geq 2$, then $\mathrm{codim}S(\Lambda')\geq 2$. By Riemann extension theorem, $f$ can extend to the points of $S(\Lambda)\cup S(\Lambda')$. Switching the role of $\Lambda$ and $\Lambda'$, we can deduce that $f$ is invertible. 

If $\mathrm{codim}S(\Lambda)= 1$, then $S(\Lambda)=S(\Lambda')=\{s=0\}$ for some reduced function $s$. Condition $4$ in Definition \ref{def:assoNambu} implies that
\begin{equation*}
\Lambda=s\Lambda_1,\; \Lambda'=s\Lambda'_1,
\end{equation*}
where $\mathrm{codim}S(\Lambda_1)\geq 2$ and $\mathrm{codim}S(\Lambda'_1)\geq 2$. The proof of the uniqueness is done by considering $\Lambda_1$ and $\Lambda'_1$ instead of $\Lambda$ and $\Lambda'$.

Now, if $\Pi$ is a tangent Nambu structure of $\mathcal{F}$, then we can write
\begin{equation*}
\Pi=f\Lambda,
\end{equation*}
for some holomorphic function $f$ defined outside the singular set $S(\Lambda)$. The function $f$ can extend to the points of $S(\Lambda)$ in the case $\mathrm{codim}S(\Lambda)\geq 2$. If  $\mathrm{codim}S(\Lambda)= 1$, the $S(\Lambda)=S(\cF)$. We write
\begin{equation*}
\Lambda=s\Lambda_1,
\end{equation*}
where $s$ is a reduced function and $\mathrm{codim}S(\Lambda_1)\geq 2$. By Condition 2 in Definition \ref{def:assoNambu}, we have $S(\cF)\subset S(\Pi)$. Hence, $\Pi$ can be written as
\begin{equation*}
\Pi=s\Pi_1,
\end{equation*}
where $\Pi_1$ is a Nambu structure. We still have
\begin{equation*}
\Pi_1=f_1\Lambda_1.
\end{equation*}
But now $f_1$ is a holomorphic function defined outside the singular set of $\Lambda_1$ which satisfies $\mathrm{codim}S(\Lambda_1)\geq 2$. This finishes the proof by using the Riemann extension theorem for $f_1$.
\end{proof}


\subsection{Associated foliations}
We say that a vector field $X$ is tangent to a Nambu structure $\Lambda$ if 
\begin{equation*}
X\wedge\Lambda=0.
\end{equation*} 
The set of tangent vector fields forms an integrable distribution and hence defines a singular foliation. However, the foliation defined in this way may lose many singularities of $\Lambda$, as in the following example shows:
\begin{example}\label{ex:2.4}
Consider the following Nambu structure on $\mathbb{C}^2$:
\begin{equation*}
\Lambda=x\frac{\partial}{\partial x}\wedge\frac{\partial}{\partial y}.
\end{equation*}
Then $\mathrm{codim} S(\Lambda)=\mathrm{codim}\{x=0\}=1$, but the foliation $\mathcal{F}$ defined by the tangent vector fields of $\Lambda$ consists of just one leaf, which is the whole space $\mathbb{C}^2$.
\end{example}

This leads us to the following definition:

\begin{defn}
A vector field $X$ is called a {\bf conformally invariant tangent} (CIT) vector field of a Nambu structure $\Lambda$ if $X$ is tangent to $\Lambda$ and $X$ conformally preserves $\Lambda$, i.e. 
\begin{equation*}
\cL_X\Lambda=f\Lambda
\end{equation*} 
for some function $f$. 
\end{defn} 

The set of CIT vector fields of $\Lambda$ will be denoted by $\CIT(\Lambda)$. The following proposition shows that $\CIT(\Lambda)$ generates an integrable singular distribution. Therefore, it defines a foliation, which will be called the {\bf associated foliation} of $\Lambda$.

\begin{prop}
The set of CIT vector fields of a smooth (resp. analytic) Nambu structure $\Lambda$ of order $q$ generates a smooth (resp. analytic) integrable singular distribution.
\end{prop} 
\begin{proof}
Let us prove in the smooth category (the analytic category is similar and simpler). First of all, we show that $\CIT(\Lambda)$ is a module over the ring of smooth functions. Assume that $X$ and $Y$ are two CIT vector fields. Then it is clear that $X+Y$ is also a CIT vector field. If $f$ is a smooth function, we need to show that $fX$ is again a CIT vector field of $\Lambda$.

Let $\Omega$ be a smooth volume form and $\omega=i_\Lambda\Omega$ be the dual $(n-q)$-form of $\Lambda$. We claim that a vector field $Z$ conformally preserves $\Lambda$ if and only if it conformally preserves $\omega$. Indeed, if $\cL_{Z}\Lambda=g\Lambda$ then
\begin{align*}
\cL_{Z}\omega&=\cL_{Z}\left(i_\Lambda\Omega\right)=i_{\cL_{Z}\Lambda}\Omega+i_{\Lambda}\cL_{Z}\Omega\\
&= gi_{\Lambda}\Omega+hi_\Lambda\Omega=(g+h)\Omega,
\end{align*}
for some smooth function $h$. Conversely, if $\cL_{Z}\omega=g'\omega$, we write
\begin{equation*}
\Lambda=\omega\lrcorner\Omega^{-1},
\end{equation*}
where $\Omega^{-1}$ is a $n$-vector satisfying $\Omega(\Omega^{-1})=1$. Then we have
\begin{align*}
\cL_{Z}\Lambda&=(\cL_{Z}\omega)\lrcorner\Omega^{-1}+\omega\lrcorner\cL_{Z}\Omega^{-1}\\
&=g'\omega\lrcorner\Omega^{-1}+h'\omega\lrcorner\Omega^{-1}=(g'+h')\Lambda.
\end{align*}
for some smooth function $h'$. 

Now if $X$ is a CIT vector field of $\Lambda$, then $i_X\omega=0$. Consequently, 
\begin{align*}
\cL_{fX}\omega&=i_{fX}d\omega+d(i_{fX}\omega)
=fi_{X}\omega\\ &=f\cL_{X}\omega.
\end{align*}
It implies that $fX$ is also a CIT vector field of $\Lambda$.

The second step is to prove that $\CIT(\Lambda)$ is involutive. Let $X$, $Y$ be two CIT vector fields of $\Lambda$. Locally near a regular point $x$ of $\Lambda$, there is a local system of coordinates $(x_1,\hdots,x_n)$ such that
\begin{equation*}
\Lambda=\frac{\partial}{\partial x_1}\wedge\hdots\wedge\frac{\partial}{\partial x_q}.
\end{equation*}
Since $X$, $Y$ are tangent to $\Lambda$, we can write
\begin{equation*}
X=\sum_{i=1}^k a_i\frac{\partial}{\partial x_i},\;Y=\sum_{i=1}^k b_i\frac{\partial}{\partial x_i}. 
\end{equation*}
It follows that locally near every regular point $x$ we have $[X,Y]\wedge\Lambda=0$. Therefore, $[X,Y]\wedge\Lambda=0$ everywhere. Moreover, if $\cL_X\Lambda=g_X\Lambda$ and $\cL_Y\Lambda=g_Y\Lambda$, then we have
\begin{align*}
\cL_{[X,Y]}\Lambda&=\cL_{X}\cL_{Y}\Lambda-\cL_{Y}\cL_{X}\Lambda\\
&=(X(g_Y)-Y(g_X))\Lambda.
\end{align*}
Hence, $[X,Y]$ is also a CIT vector field. 

Now, since every CIT vector field conformally preserves $\Lambda$, its local flow also preserves $\Lambda$ conformally. Consequently, it preserves everything conformally generated by $\Lambda$, including the singular distribution generated by $\CIT(\Lambda)$. Therefore, by Stefan-Sussmann theorem (see Chapter 1 of \cite{DufourZung-PoissonBook}), we obtain that the smooth singular distribution generated by $\CIT(\Lambda)$ is integrable.
\end{proof}

\begin{defn}
The foliation generated by the CIT vector fields of $\Lambda$ is called the {\bf associated foliation} of $\Lambda$, and denoted by $\cF_\Lambda$. 
\end{defn}

Clearly, if $f$ is an invertible function (i.e. $f \neq 0$ everwhere) then 
$\cF_\Lambda = \cF_{f\Lambda}$. 
In Example \ref{ex:2.4}, $\cF_\Lambda$ is generated by $\{x\frac{\partial}{\partial x},\frac{\partial}{\partial y}\}$ and consists two leaves $\{x=0\}$
and $\{x\neq 0\}$.

\begin{prop}\label{prop:NambuTangentToAssoFol}
For any Nambu structure $\Lambda$, the singular set of $\Lambda$ contains the singular set of its associated foliation:
 $$S(\mathcal{F}_\Lambda)\subset S(\Lambda).$$
\end{prop}
\begin{proof}
Fix an arbitrary local coordinate system  $(x_1,\hdots,x_n)$. Let $P$ be a regular point  of $\Lambda$ in this coordinate system. Then we may assume, without lost of generality, that
$$a_{12\hdots q}(P)\neq 0,$$
where
$$\Lambda=\sum_{i_1<\cdots<i_q}a_{i_1\hdots i_q}\frac{\partial}{\partial x_{i_1}}\wedge\hdots\wedge \frac{\partial}{\partial x_{i_q}}.$$
Denote $X_i=(dx_1\wedge\hdots\wedge\widehat{dx_i}\wedge\hdots\wedge dx_q)\lrcorner\Lambda$ (the hat means that the corresponding term is missing in the expression). Then $X_i$ is local Hamiltonian vector fields of $\Lambda$, and obviously it is in $\CIT(\Lambda)$. The condition $a_{12\hdots q}(P)\neq 0$ implies that 
$$X_1\wedge\hdots\wedge X_q(P)\neq 0,$$
and $\cF_\Lambda$ is generated by $X_1,\hdots,X_q$ near $P$. So $P$ is also a regular point of $\cF_\Lambda$.
\end{proof}

\begin{prop}\label{prop:Tang&Asso}
Let $\Lambda$ be a holomorphic Nambu structure. If $\mathrm{codim}(\Lambda)\geq 2$, then every tangent vector of $\Lambda$ is also a conformally invariant vector field of $\Lambda$.  
\end{prop}
\begin{proof}
Let $X$ be a tangent vector of $\Lambda$. Locally near a regular point of $\Lambda$, there is a local system of coordinate $(x_1,\hdots,x_n)$ such that
\begin{equation*}
\Lambda=\frac{\partial}{\partial x_1}\wedge\hdots\wedge\frac{\partial}{\partial x_1},\,X=\sum_{i=1}^k a_i\frac{\partial}{\partial x_i}.
\end{equation*}
It leads to 
\begin{equation*}
\cL_X\Lambda=g\Lambda,
\end{equation*}
where $g$ is a holomorphic functions defined outside the singular set of $\Lambda$. Since $\mathrm{codim}(\Lambda)\geq 2$, by Riemann extension theorem $g$ can extend to the points of $S(\Lambda)$ and it implies that $X$ is CIT vector field.
\end{proof}

\subsection{From foliations to Nambu structures and back}
Given a local Nambu structure $\Lambda$ and a local function $f$, we will say that $\Lambda$ \emph{conformally preserves $f$} if there exists a local Nambu structure $\Sigma$ of order $q-1$ such that
\begin{equation*}
[f,\Lambda]=f\Sigma,
\end{equation*} 
where the above bracket means the Schouten bracket.
The set of functions which are  conformally preserved by  $\Lambda$ will be denoted by $\mu(\Lambda)$.
\begin{lemma}\label{lem:210} If a function $f$ is irreducible and $f\nmid\Lambda$, then $f\in\mu(\Lambda)$ if and only if   
\begin{align*}
f|X(f),\forall X\in\CIT(\Lambda).
\end{align*}
\end{lemma}
\begin{proof}
If $Y$ is tangent to $\Lambda$, then 
\begin{equation*}
Y(f)\Lambda-Y\wedge[f,\Lambda]=[f,Y\wedge\Lambda]=0.
\end{equation*}
Therefore, $[f,\Lambda]=f\Sigma$ if and only if $f|Y(f)$ for all $Y$ tangent to $\Lambda$. Hence, it remains to show that if $f|X(f)$ for all $X\in\CIT(\Lambda)$ then   $f|Y(f)$ for all $Y$ tangent to $\Lambda$. 

Let $Y$ be a tangent vector of $\Lambda$. Using the same argument as in the proof of Proposition \ref{prop:Tang&Asso}, we have
\begin{equation*}
\cL_Y\Lambda=g\Lambda,
\end{equation*}
where $g$ is a holomorphic functions defined outside the singular set of $\Lambda$. The function $g$ is  meromorphic, so we can write 
\begin{equation*}
g=\frac{t}{s},
\end{equation*}
where $\{s=0\}\subset S(\Lambda)$. We claim that $sY\in\CIT(\Lambda)$. Indeed, let $\Omega$ be a  volume form and $\omega=i_{\Lambda}\Omega$, we have
\begin{align*}
\cL_{Y}\omega&=i_{\cL_Y\Lambda}\Omega+i_\Lambda\cL_Y\Omega=gi_\Lambda\Omega+hi_\Lambda\Omega\\
&=\frac{t+hs}{s}\omega,
\end{align*}
for some holomorphic function $h$. This leads to
\begin{align*}
\cL_{s Y}\Lambda&=(\cL_{s Y}\omega)\lrcorner\Omega^{-1} + \omega\lrcorner\cL_{sY}\Omega^{-1}\\
&=s\cL_Y\omega\lrcorner\Omega^{-1}+h'\omega\lrcorner\Omega^{-1}\\
&=(t+hs+h')\Lambda,
\end{align*}
for some holomorphic function $h'$. Therefore, $sY\in\CIT(\Lambda)$. It implies that
$$f|sY(f).$$
Since $f\nmid\Lambda$ and $\{s=0\}\subset S(\Lambda)$, we have that $f$ and $s$ are coprime. Consequently, $f|Y(f)$.
\end{proof}

The following proposition shows how does a Nambu structure change after the cycle: $\Lambda\rightarrow\cF_\Lambda\rightarrow\Lambda'$. 
\begin{prop}\label{prop:NambuFolNambu}
Let $\Lambda$ be a holomorphic Nambu structure and $\cF_\Lambda$ be its associated foliation. Suppose that $\Lambda'$ is an associated Nambu structure of $\cF_\Lambda$. 
\begin{itemize}
\item{} If $\mathrm{codim}S(\Lambda)=2$, then $\Lambda'=u\Lambda$ for some invertible function $u$.
\item{} If $\Lambda=\prod f_i^{m_i}\prod g_j^{m_j} \Lambda_1$, where $\mathrm{codim}S(\Lambda_1)\geq 2$, $f_i,g_j$ are irreducible, $f_i\in\mu(\Lambda)$, $g_j\not\in\mu(\Lambda)$, then $\Lambda'=u\prod g_j\Lambda_1$ for some invertible function $u$.
\end{itemize}
\end{prop}
\begin{proof}
First consider the case $\mathrm{codim}S(\Lambda)\geq 2$. By Proposition \ref{prop:NambuTangentToAssoFol},  $S(\cF_\Lambda)\subset S(\Lambda)$. Consequently, $\mathrm{codim}S(\cF_\Lambda)\geq 2$ and this implies that $\Lambda$ is an associated Nambu structure of $\cF$. By the uniqueness of associated Nambu structure, $\Lambda=u\Lambda'$ for some invertible function $u$. 

The case $\mathrm{codim}S(\Lambda)=1$ is a direct consequence of Lemma \ref{lem:NambuFolNambu} below. \end{proof}
\begin{lemma}\label{lem:NambuFolNambu}
Let $f^m\Lambda$ be a holomorphic Nambu structure and $\cF_{f^m\Lambda}$ be its associated foliation, where $f$ is an irreducible function and $f\nmid\Lambda$. Suppose that $\Lambda'$ is an associated Nambu structure of $\cF_{f^m\Lambda}$. 
\begin{itemize}
\item{} If $f\in\mu(\Lambda)$, then $\Lambda=g\Lambda'$ for some holomorphic function $g$.
\item{} If $f\notin\mu(\Lambda)$, then $\Lambda'=f\Lambda''$ and $\Lambda=g\Lambda''$ for some holomorphic function $g$. 
\end{itemize}
\end{lemma}
\begin{proof}
 We first show that $\CIT(f^m\Lambda)\subset\CIT(\Lambda)$. If $X\in\CIT(f^m\Lambda)$, then there is a holomorphic function $h$ such that 
\begin{equation*}
\cL_{X}(f^m\Lambda)=hf^m\Lambda.
\end{equation*}
It leads to
\begin{equation*}
mX(f)\Lambda+f\cL_X\Lambda=hf\Lambda.
\end{equation*}
Using the same argument as in the proof of Proposition \ref{prop:Tang&Asso}, we have 
$$\cL_X\Lambda=s\Lambda,$$
where $s$ is a holomorphic function defined outside $S(\Lambda)$. Since $f\nmid\Lambda$, $X(f)$ must vanish on zero locus of $f$. It implies that $f|X(f)$ and hence $X\in\CIT(\Lambda)$.
Let $\cF_\Lambda$ be the associated foliation of $\Lambda$.
 
If $f\in\mu(\Lambda)$, we claim that 
$\cF_{f^m\Lambda}=\cF_{\Lambda}$. 
Indeed, suppose that $X\in\CIT(\Lambda)$, then there is a function $h'$ such that 
$$\cL_X(\Lambda)=h'\Lambda.$$
Hence, $$\cL_X(f^m\Lambda)=\left(h'+\frac{m\cL_X(f)}{f}\right)f^m\Lambda.$$
Since  $\cF_{f^m\Lambda}=\cF_{\Lambda}$, $\Lambda'$ is also an associated Nambu structure of $\cF_\Lambda$ whose singular set doesn't contain $\{f=0\}$. It implies that  $\Lambda=g\Lambda'$ for some holomorphic function $g$.  

Now consider the case $f\notin\mu(\Lambda)$. Let $Y\in\CIT(\Lambda)$ such that $f\nmid Y(f)$ then $Y\not\in\CIT(f^m\Lambda)$. We claim that at every point $P$ which  satisfies $Y(P)\neq 0$ and $f(P)=0$ is a singularity of $\cF_{f^m\Lambda}$. Indeed, choose a coordinate system $(x_1,\hdots,x_q)$ at $P$ such that $Y=\frac{\partial}{\partial x_1}$ and 
$$\Lambda=\frac{\partial}{\partial x_1}\wedge\hdots\wedge\frac{\partial}{\partial x_q}.$$
Suppose that $\cF_{f^m\Lambda}$ is regular at $P$. Because $\cF_{\Lambda}$ is a saturation of $\cF_{f^m\Lambda}$, we must have $\cF_{f^m\Lambda}=\cF_{\Lambda}$ locally near $P$. Therefore, $\{x_2=c_2,\cdots,x_n=c_n\}$ are invariant curves of $\cF_{f^m\Lambda}$. It implies that $a\frac{\partial}{\partial x_1}\in\CIT(\cF_{f^m\Lambda})$ for some function $a$. But this can not happen when $f\nmid\frac{\partial f}{\partial x_1}$.  In conclusion, $$\{P:f(P)=0, Y(P)\neq 0\}\subset S(\cF_{f^m\Lambda})$$ 
Since $f\nmid Y$ and $S(\cF_{f^m\Lambda})$ is closed, we have  $\{f=0\}\subset S(\cF_{f^m\Lambda})$ and hence $\{f=0\}\subset S(\Lambda')$.  Therefore, $\Lambda'=f\Lambda''$ for some Nambu structure $\Lambda''$. It is easy to show that $\Lambda=g\Lambda''$ for some holomorphic function $g$. \end{proof}

\begin{cor}\label{cor:DecomNambu}
Let $\Lambda$ be a holomorphic Nambu structure and $\mathcal{F}$ be its associated foliation. Then locally near any point $x$, we can write
\begin{equation*}
\Lambda=g\cdot\frac{\partial}{\partial x_1}\wedge\hdots\wedge\frac{\partial}{\partial x_k}\wedge\Pi,
\end{equation*}
where $\{x_{k+1}=\cdots=x_{n}=0\}$ is the leaf of $\mathcal{F}$ passing through $x$, $g$ is a holomorphic function and $\frac{\partial}{\partial x_i}$, $i=1,\hdots,k$, commutes with $\Pi$.  
\end{cor}
\begin{proof}
Choose a system of coordinates $(x_1,\hdots,x_n)$ such that the leaf of $\cF$ passing through $x$ is the $k$-dimensional disk $\{x_{k+1}=\cdots=x_n=0\}$ and each $k$-dimensional disk $\{x_{k+1}=c_{k+1},\hdots, x_n=c_n\}$ is wholly contained in some leaf of $\cF$. Denote by $\cF_0$ the restriction of $\cF$ of the disk $\{x_1=\cdots=x_k=0\}$. Let $\Pi$ be a associated Nambu structure of $\cF_0$. Then, it is clear that for each $i=1,\hdots,k$, $\partial x_i$ commutes with $\Pi$ and
\begin{equation*}
\Lambda'=\frac{\partial}{\partial x_1}\wedge\hdots\wedge\frac{\partial}{\partial x_k}\wedge\Pi,
\end{equation*}
be an associated Nambu structure of $\cF$. By Proposition \ref{prop:NambuFolNambu}, we have $\Lambda=g\Lambda'$ for some function $g$.
\end{proof}


\begin{prop}
Let $\cF$ be a holomorphic singular foliation and $\Lambda$ be its associated Nambu structure. Suppose that $\cF_\Lambda$ is an associated foliation of $\Lambda$ then $\cF_\Lambda$ is a saturation of $\cF$. Moreover, if $\mathrm{codim}S(\cF)\geq 2$ then $\mathrm{codim}S(\cF_\Lambda)\geq 2$. 
\end{prop}
\begin{proof}
If $\mathrm{codim}S(\cF)\geq 2$, then $\mathrm{codim}S(\Lambda)\geq 2$. By Proposition \ref{prop:NambuTangentToAssoFol}, we have $\mathrm{codim}S(\cF_\Lambda)\geq 2$. Let us now show that $\cF_\Lambda$ is a saturation of $\cF$. Fix a point $P$ of the foliation and assume that $X$ is a tangent vector field of $\cF$ which satisfies $X(P)\neq 0$. We will show that $X$ is also a tangent vector of $\cF_\Lambda$. Choose a coordinate system $(x_1,\hdots,x_n)$ at $P$ such that $X=\frac{\partial}{\partial x_1}$ and every disk $\{x_2=c_2,\hdots,x_n=c_n\}$ is contained in a leaf of $\cF$. Denote by $\cF_0$ the restriction of $\cF$ on $\{x_1=0\}$. Suppose $\Pi$ an associated Nambu structure of $\cF_0$, then $\Lambda'=\frac{\partial}{\partial x_1}\wedge\Pi$ is an associated Nambu structure of $\cF$. It is clear that $\frac{\partial}{\partial x_1}$ is a CIT of $\Lambda'$ and hence a CIT of $\Lambda$. Therefore, $X$ is tangent to $\cF_\Lambda$.    
\end{proof}

\section{Commuting Nambu tensors}

\subsection{The case $q_1 + q_2 \leq n$}

\begin{defn} \label{defn:CommutingNambu1}
 Let $\Lambda_1$ and $\Lambda_2$ be Nambu tensors of orders $q_1$ and $q_2$ respectively on a manifold $M$ of dimension $n$, 
such that $q_1 + q_2 \leq n$ and $\Lambda_1 \wedge \Lambda_2 \neq 0$ almost everywhere.
Then we will say that $\Lambda_1$ and $\Lambda_2$ {\bf commute} with each other 
if their Schouten bracket vanishes:
\begin{equation}
[\Lambda_1, \Lambda_2] = 0. 
\end{equation}
\end{defn}

\begin{prop}[\cite{DufourZung-PoissonBook}] \label{prop:CommutingNambu1}
 Let $\Lambda_1$ and $\Lambda_2$ be two commuting 
Nambu tensors of orders $q_1$ and $q_2$ respectively on a manifold $M$ of dimension $n$, with $q_1 + q_2 \leq n.$ Suppose that
$\Lambda_1 (O) \wedge \Lambda_2 (O) \neq 0$ at a point $O \in M.$ Then $\Lambda_1$ and $\Lambda_2$ can be put into the 
following simultaneous normal form with respect to a local coordinate system $(x_1,\hdots,x_n)$ in a neighborhod of  $O$:
\begin{equation} \label{eqn:L1L2}
 \begin{array}{l}
  \Lambda_1 = \frac{\partial}{\partial x_1} \wedge \hdots \wedge \frac{\partial}{\partial x_{q_1}}, \\
 \Lambda_2 = \frac{\partial}{\partial x_{q_1+1}} \wedge \hdots \wedge \frac{\partial}{\partial x_{q_1 + q_2}}.
 \end{array}
\end{equation}
\end{prop}

\begin{proof}
The above proposition was mentioned without proof in \cite{DufourZung-PoissonBook}, so for the completeness of
exposition let us present here a proof. 

Let us first consider the case $n = q_1 + q_2$.  In this case, we have two local transverse foliations near $O$, generated by
$\Lambda_1$ and $\Lambda_2$ respectively. We can find a coordinate system $(y_1,\hdots,y_{q_1},z_1,\hdots,z_{q_2})$
near $O$ such that the foliation generated by $\Lambda_1$ is $\{z_1= const.,\hdots,z_{q_2} = const.\},$ and 
 the foliation generated by $\Lambda_1$ is $\{y_1= const.,\hdots,y_{q_1} = const.\}.$ In other words,
$\Lambda_1 = f_1 \frac{\partial}{\partial y_1} \wedge \hdots \wedge  \frac{\partial}{\partial y_{q_1}}$ and
$\Lambda_2 = f_2 \frac{\partial}{\partial z_1} \wedge \hdots \wedge  \frac{\partial}{\partial z_{q_2}},$ where
$f_1$ and $f_2$ are some functions. The equality $[\Lambda_1,\Lambda_2] = 0$ implies that $f_1$ (resp. $f_2$)
does not depend on the variables $z_1,\hdots, z_{q_2}$ (resp. $y_1,\hdots, y_{q_1}$). One can then find functions
$x_1,\hdots,x_{q_1}$ (resp. $x_{q_1+1},\hdots, x_{n}$) which depend only on the coordinates 
$y_1,\hdots,y_{q_1}$ (resp. $z_{1},\hdots, z_{q_2}$), such that $\Lambda_1$ and $\Lambda_2$ have the canonical
form \eqref{eqn:L1L2} in the new coordinate system $(x_1,\hdots, x_{n}).$

The case when $n > q_1 + q_2$  can be reduced to a parametrized version of the case with the dimension equal to
$q_1 + q_2.$ The main point is to  prove that $\Lambda_1 \wedge \Lambda_2$ is a Nambu structure. Locally we can write
\begin{equation}
 \begin{array}{l}
  \Lambda_1 = X_1 \wedge \hdots \wedge X_{q_1}, \\
 \Lambda_2 = Y_1 \wedge \hdots \wedge Y_{q_2},
 \end{array}
\end{equation}
where the vector fields $X_1,\hdots,X_{q_1},Y_1,\hdots,Y_{q_2}$ are linearly independent can can be completed
by vector fields $Z_1,\hdots,Z_{q_3}$, where $q_3 = n - q_1 - q_2,$ to become a basis field for the tangent bundle of $M$  near $O$.
The fact that $\Lambda_1$ is a Nambu structure means that $X_1,\hdots,X_{q_1}$ satisfy the Frobenius integrability condition,
i.e. $[X_i,X_j] (x)$ lies in the linear span of $X_1(x),\hdots, X_{q_1}(x)$ for any $x$ near $O$. The same holds for the
vector fields $Y_1,\hdots,Y_{q_2}.$ The equality
\begin{multline}
 0 = [\Lambda_1,\Lambda_2] = \sum_{i,j} ((-1)^{i+j} [X_i,Y_j]  \wedge X_1 \wedge \hdots \wedge X_{i-1} \wedge X_{i+1}\wedge
\hdots \wedge X_{q_1} \\
\wedge Y_{1} \wedge \hdots \wedge Y_{j-1} \wedge Y_{j+1} \wedge \hdots \wedge Y_{q_2})
\end{multline}
implies that $[X_i,Y_j]$ must also lie in the span of $X_1,\hdots,X_{q_1}, Y_1,\hdots, Y_{q_2},$ because it cannot contain a component
of the type $Z_k$ in its decomposition in the basis $(X_i,Y_j,Z_k).$ It means that $X_1,\hdots,X_{q_1},Y_1,\hdots,Y_{q_2}$ span an 
integrable distribution, which is tangent to a $(q_1+q_2)$-dimensional foliation, and $\Lambda_1 \wedge \Lambda_2$ is a contravariant
volume form on the leaves of this foliation,. Thus $\Lambda_1 \wedge \Lambda_2$ is a Nambu structure.
\end{proof}

\begin{prop} \label{prop:CommutingNambu2}
 Let $\Lambda_1, \hdots, \Lambda_s$ be pairwise commuting 
Nambu tensors of orders $q_1, \hdots, q_s$ respectively on a manifold $M$ of dimension $n$, with $q_1 + \hdots + q_s \leq n.$ Suppose that
$\Lambda_1 (O) \wedge \hdots \wedge  \Lambda_s (O) \neq 0$ at a point $O \in M.$ Then $\Lambda_1, \hdots, \Lambda_s$ can be put into the 
following simultaneous normal form with respect to a local coordinate system $(x_1,\hdots,x_n)$ in a neighborhod of  $O$:
\begin{equation}
 \begin{array}{l}
  \Lambda_1 = \frac{\partial}{\partial x_1} \wedge \hdots \wedge \frac{\partial}{\partial x_{q_1}}, \\  
\Lambda_2 = \frac{\partial}{\partial x_{q_1+1}} \wedge \hdots \wedge \frac{\partial}{\partial x_{q_1 + q_2}}, \\
\hdots \\
\Lambda_s = \frac{\partial}{\partial x_{q_1+ \hdots + q_{s-1} + 1}} \wedge \hdots \wedge \frac{\partial}{\partial x_{q_1 + \hdots + q_s}}. \\
 \end{array}
\end{equation}
\end{prop}
 
\begin{proof}
By induction. Apply Proposition \ref{prop:CommutingNambu1} to 
$\Lambda_1$ and $\Pi_2 = \Lambda_2 \wedge \hdots \wedge \Lambda_s,$
we get a coordinate system in which $\Lambda_1$ and $\Pi_2$ are normalized. The problem is then reduced
to (a parametrized version of) the problem of normalization of the $(s-1)$-tuple of Nambu structures
$\Lambda_2,\hdots, \Lambda_s.$
\end{proof}

\begin{example}
 Consider a non-identity 
linear automorphism $\phi$ from a torus $\bbT^n$ to itself, and denote by $M$ its suspension: $M$ is a torus fibration over the circle
$\bbS^1,$ with a ``horizontal'' vector field $X$ which is a lifting of the standard constant vector field on $\bbS^1$ such that the Poincaré
map of $X$ on a fiber of $M$ is isomorphic to $\phi.$ Denote by $\Lambda$ the standard contravariant volume form on the torus fibers
of $M$. Then $\Lambda$ is a Nambu structure on $M$ which is preserved by $X$. We can also view $X$ as a Nambu structure of order 1
on $M$. Then $\Lambda$ and $X$ are two transverse commuting Nambu structures on $M$. Notice that the holonomy of the foliation 
generated  by $X$ near the closed orbits of $X$ are not trivial on the tori (i.e. the leaves of $\Lambda$), i.e. we have here two transverse
foliations which are generated by two commuting Nambu structures, but which do not satisfy 
Movshev's holonomy condition \cite{Movshev-YM2009}. Nevertheless, this example is of almost direct product type, and it is reasonable
to consider almost direct products of manifolds as examples of commuting foliations. 
\end{example}

\subsection{Reduction of Nambu structures}

Before treating the case $q_1 + q_2 > n,$ let us make a digression and discuss briefly about the reduction of Nambu structure, because
we will reduce the case with $q_1 + q_2 > n,$ to the case with $q_1 + q_2 = n.$

A vector field $X$ is called a {\bf Nambu vector field} with respect to a Nambu structure $\Lambda$ if $X$ preserves $\Lambda,$ i.e.
$\cL_X \Lambda = [X,\Lambda] = 0,$ where $\cL$ denotes the Lie derivation, and the bracket is the Schouten bracket. $X$ is called a
{\bf Hamiltonian vector field} with respect to $\Lambda$ if there are $q-1$ functions $f_1,\hdots, f_{q-1}$ such that
\begin{equation}
 X = (df_1 \wedge \hdots \wedge df_{q-1}) \lrcorner \Lambda.
\end{equation}
Any Hamiltonian vector field is a Nambu vector field which is tangent to the foliation generated by $\Lambda.$ Conversely, a Nambu vector field
which is tangent to the foliation of the Nambu tensor $\Lambda$ is a locally Hamiltonian vector field near each non-singular point of $\Lambda.$
More generally, if $f_1,\hdots,f_{q-k}$ are $q-k$ functions, where $1 \leq k < q,$ and $\Lambda$ is a Nambu tensor of order $q,$ then the $k$-vector field
\begin{equation}
\Pi_{f_1,\hdots,f_{q-k}} = (df_1 \wedge \hdots \wedge df_k) \lrcorner \Lambda 
\end{equation}
is a Nambu structure of order $k$ which will be called a {\bf Hamiltonian Nambu structure} with respect to $\Lambda:$ the foliation of 
$\Pi_{f_1,\hdots,f_{q-k}}$ is tangent to the foliation of $\Lambda,$ and $\Pi_{f_1,\hdots,f_{q-k}}$ ``preserves'' $\Lambda$ in the sense that
\begin{equation}
 [\Pi_{f_1,\hdots,f_{q-k}},\Lambda] = 0.
\end{equation}

Assume $\Lambda$ is a given Nambu structure of order $q$, and $\Pi$ is a regular Nambu structure of order $k$ $(1 \leq k < q)$ on a manifold $M$
of dimension $n$, with the following properties: \\
i) The leaf space $M/\cF^\Pi$ of the regular foliation $\cF^\Pi$ of $\Pi$  in $M$ is a Hausdorff
$(n-k)$-dimensional manifold. \\
ii) The foliation of $\Pi$ is tangent to the foliation of $\Lambda,$ and $[\Pi,\Lambda] = 0.$
Then there is a unique Nambu structure $\Theta$ of order $q-k$ on the quotient manifold $M/\cF^\Pi$ (the leaf space of $\cF^\Pi$), which is
the {\bf reduction of $\Lambda$ by $\Pi$} in the following sense: 
locally near each point $z \in M$ there is a  coordinate system $(x_1,\hdots, x_n)$ in which
\begin{equation}
\Pi = \frac{\partial}{\partial x_1} \wedge \hdots \wedge  \frac{\partial}{\partial x_k},
\end{equation}
the variables $(x_{k+1},\hdots,x_n)$ are local coordinate system on the quotient manifold $M/\cF^\Pi,$ 
\begin{equation}
\Lambda =  \Pi \wedge \Theta
\end{equation}
 and the expression of $\Theta$ involves only the variables $(x_{k+1},\hdots,x_n).$ The above reduction process is an imitation of the
reduction of Poisson structures. It can be done locally, i.e. we can talk about the {\bf local reduction} of $\Lambda$ by $\Pi$
near any given regular point of $\Pi$ (without the need of the assumption that $\Pi$ is
globally regular).

\subsection{The case $q_1 + q_2 > n$}

When $q_1 + q_2 \geq n+2$ then we always have $[\Lambda_1,\Lambda_2] = 0$ for any $q_1$-vector field $\Lambda_1$ and
$q_2$-vector field $\Lambda_2$, and we have to change the definition of commutativity in this case in order for it to be
meaningful. When $q_1 + q_2 = n+1$ then the condition is non-trivial, but not sufficient to imply that $\Lambda_1$ and
$\Lambda_2$ can be put into a constant form simultaneously near a non-singular point. The best that we can have
when $q_1 + q_2 = n+1$ under the condition $[\Lambda_1,\Lambda_2] = 0$ is the following:

\begin{prop} \label{prop:CommutingNambu3}
 Let $\Lambda_1$ and $\Lambda_2$ be two 
Nambu tensors of orders $q_1$ and $q_2$ respectively on a manifold $M$ of dimension $n$, such that $q_1 + q_2 =  n+1$, 
  $[\Lambda_1,\Lambda_2] = 0$ and $(\Lambda_1\lrcorner\Omega)(O)\wedge(\Lambda_2\lrcorner\Omega)(O) \neq 0$ at a point $O \in M$, where $\Omega$ is a local volume form. 
Then in a neighborhood of $O$ there exists a vector field $X$ which is locally Hamiltonian with respect to both
$\Lambda_1$ and $\Lambda_2$. Conversely, if $\Lambda_1$ and $\Lambda_2$ are two 
Nambu tensors of orders $q_1$ and $q_2$ respectively on a manifold $M$ of dimension $n = q_1 + q_2 - 1$,  which admit
a common Hamiltonian vector field $X,$ then $[\Lambda_1,\Lambda_2] = 0.$
\end{prop}

\begin{proof}
 Denote by $\cF_1$ and $\cF_2$ the foliations of $\Lambda_1$ and $\Lambda_2$ respectively. Then the intersection of $\cF_1$
with $\cF_2$ near $O$ is a regular 1-dimensional foliation. Let $\tilde X$ be a local vector field which is tangent to this intersection
foliation, i.e. $\tilde X$ is tangent to both $\Lambda_1$ and $\Lambda_2:$ $\tilde X \wedge \Lambda_1 = \tilde X \wedge \Lambda_2 = 0.$ Since
$\tilde X$ is tangent to $\Lambda_1,$ we have $[\tilde X,\Lambda_1] = a \Lambda_1$ for some function $a.$ By putting $X = f \tilde X,$ where
$f$ is a local solution of the ordinary differential equation $\tilde X(f) = af$, we get $[X,\Lambda] = 0,$ i.e. $X$ is a locally
Hamiltonian vector field of $\Lambda_1.$ Locally near $O$ we can write $\Lambda_1 = X \wedge \Pi_1,$ and also
$\Lambda_2 = X \wedge \Pi_2,$ where $\Pi_1$ and $\Pi_2$ are Nambu structures and $\Pi_1$
is invariant with respect to $X.$ The equality $0 = [\Lambda_1,\Lambda_2] = \pm [X,\Lambda_2] \wedge \Pi_1$
implies that $ [X,\Lambda_2] = 0,$ i.e. $X$ is also a local Hamiltonian vector field with respect to $\Lambda_2.$
\end{proof}

Remark that, in the above proposition, even though we can choose $X,\Pi_1$ and $\Pi_2$ such that $\Lambda_1 = X\wedge \Pi_1,$
$\Lambda_2 = X\wedge \Pi_2,$ and $\Pi_1$ and $\Pi_2$ are two Nambu tensors invariant with respect to $X$, we cannot
arrange so that $[\Pi_1,\Pi_2] = 0$ in general. A  way to define commutativity of two Nambu structures 
whose total rank is greater than the dimension of the
manifold is as follows:

\begin{defn} \label{defn:CommutingNambu2}
Let $\Lambda_1$ and $\Lambda_2$ be two Nambu structures of orders $q_1$ and $q_2$ respectively on a manifold $M$ of dimension
$n$, such that $q_1 + q_2 - n = k > 0.$ Then we say that $\Lambda_1$ commutes with $\Lambda_2$ if in a neighborhood of any point 
$O \in M$ such that the foliations generated by $\Lambda_1$ and $\Lambda_2$ are transverse to each other near $O$, there is a
local coordinate system $(x_1,\hdots,x_n)$ such that
\begin{equation}
\begin{array}{l}
 \Lambda_1 = \frac{\partial}{\partial x_1} \wedge \hdots \wedge \frac{\partial}{\partial x_{q_1}}, \\  
\Lambda_1 = \frac{\partial}{\partial x_1} \wedge \hdots \wedge \frac{\partial}{\partial x_{k}}
\wedge  \frac{\partial}{\partial x_{q_1+1}} \wedge \hdots \wedge \frac{\partial}{\partial x_{n}}.
\end{array}
\end{equation}
\end{defn}

Another equivalent definition of commutativity of Nambu structures 
in the case $q_1 + q_2 > n$ is given by the following
proposition and by induction on $q_1 + q_2 - n$:

\begin{prop} \label{prop:CommutingNambu4}
 $\Lambda_1$ commutes with $\Lambda_2$ if and only if near any point $O$ such that $\Lambda_1$ and  $\Lambda_2$ are transverse at $O$,
there is a vector field $X$ such that $X(O) \neq 0,$ $X$ is Hamiltonian with respect to both $\Lambda_1$ and $\Lambda_2,$ and the
local reductions of $\Lambda_1$ and $\Lambda_2$ with respect to $X$ commute with each other. 
\end{prop}

The proof is straightforward and by induction on $q_1 + q_2 -n.$

In Definition \ref{defn:CommutingNambu1} for the case with $q_1 + q_2 \leq n,$ we assumed that $\Lambda_1 \wedge \Lambda_2 \neq 0.$
The case when $\Lambda_1 \wedge \Lambda_2$ is identically zero is a degenerate case, and in that case the definition of commutativity 
has to be changed as follows to make sense:

\begin{defn} Let $\Lambda_1$ and $\Lambda_2$ be two Nambu structures of order $q_1$ and $q_2$ respectively on a manifold $M$, 
and $k> 0$ is a constant, such that  $\dim (T_z \cF_1 \cap T_z \cF_2) = k$ for almost every $z \in M,$ where $\cF_1$ and $\cF_2$ denote
the foliations of $\Lambda_1$ and $\Lambda_2$ respectively. Then we will say that $\Lambda_1$ commutes with $\Lambda_2$
if near any point $z \in M$ such that $\dim (T_z \cF_1 \cap T_z \cF_2) = k$ there is a local coordinate system 
$(x_1,\hdots,x_n)$ such that
\begin{equation}
\begin{array}{l}
 \Lambda_1 = \frac{\partial}{\partial x_1} \wedge \hdots \wedge \frac{\partial}{\partial x_{q_1}}, \\  
\Lambda_1 = \frac{\partial}{\partial x_1} \wedge \hdots \wedge \frac{\partial}{\partial x_{k}}
\wedge  \frac{\partial}{\partial x_{q_1+1}} \wedge \hdots \wedge \frac{\partial}{\partial x_{q_1 + q_2 -k}}.
\end{array}
\end{equation}
\end{defn}

\begin{example}
Given an action of a Lie algebra $\mathfrak{g}$ on a manifold $M$, i.e. a Lie morphism $\mathfrak{g} \to \cX(M)$ from $\mathfrak g$
to the Lie algebras of vector fields on $M$, such that its general orbits have dimension equal to $q$, we can construct
a family of Nambu structure as follows: for each element $\xi \in \wedge^q \mathfrak{g},$ denote by $\Lambda_\xi$ the image of 
$\xi$ via the natural extension  $\wedge^q \mathfrak{g} \to \wedge^q \cX(M)$ of the map $\mathfrak{g} \to \cX(M).$ Then 
$\Lambda_\xi$ is a Nambu structure of order $q$ on $M$ for any $\xi,$ and is $\xi$ is chosen well enough then almost all the regular
orbits of the action of $\fg$ on $M$ are also the regular leaves of the foliation of $\Lambda_\xi.$ Now if there are two commuting
actions of two Lie algebras $\fg_1$ and $\fg_2$ on $M$, and two elements $\xi_1 \in \wedge^{q_1} \fg_1$ and
$\xi_2 \in \wedge^{q_2} \fg_2$, then the two associated Nambu structures $\Lambda_{\xi_1}$ and $\Lambda_{\xi_2}$ will commute
with each other. We will leave the verification of this fact as a simple exercise to the reader.
\end{example}

\section{Almost direct products}

The almost direct product example mentioned in the introduction of this note is in fact a general consruction of foliations which
are transverse to each other and have complementary dimensions. More precisely, we have the following simple proposition:

\begin{prop} \label{prop:AlmostDirectProduct}
Let $\cF_1$ and $\cF_2$ be two regular foliations on a connected compact manifold $M$, such that $T_xM = T_x\cF_1 \oplus T_x \cF_2$ for any 
$x \in M.$ Then the triple $(M,\cF_1,\cF_2)$ is isomorphic to an almost direct product model
\begin{equation}
(F_1 \times F_2) / G,
\end{equation}
where $F_1$ and $F_2$ are two connected manifolds (which are not necessarily compact), $G$ is a discrete group which acts  
on the product $F_1 \times F_2$ freely and diagonally. 
\end{prop}

The about result seems to be folkloric, but unfortunately We don't have an exact reference for it.
So for the sake of completeness, let us give here a proof of it.

\begin{proof} 
First notice that, due to the fact that $\cF_1$ and $\cF_2$ are transverse and have complementary dimensions, $\cF_1$ creates a locally
flat parallel transport among the leaves of $\cF_2$ and vice versa: given any two paths $\gamma_1$ tangent to $\cF_1$ and
$\gamma_2$ tangent to $\cF_2$ such that $\gamma_1(0) = \gamma_2(0),$ there is a unique natural 
way to transport $\gamma_1$ along  $\gamma_2$ such that $\gamma_1$ remains always tangent to $\cF_1.$
This locally flat parallel transport exists locally even if the manifold $M$ is not compact.
The theorem still holds if we replace the compactness condition by the following weaker \emph{completeness condition}:
the parallel transport exists not only locally, but also globally, i.e. given any two paths $\gamma_1$ tangent to $\cF_1$ and
$\gamma_2$ tangent to $\cF_2$ such that $\gamma_1(0) = \gamma_2(0),$ there is a unique natural 
way to transport $\gamma_1$ (resp. $\gamma_2$) along  $\gamma_2$ (resp. $\gamma_1$)
such that $\gamma_1$  (resp. $\gamma_2$) remains always tangent to $\cF_1$ (resp. $\cF_2$).

Take a point $z \in M.$ Denote by $\cF_1(z)$ (resp. $\cF_2(z)$) the leaf of $\cF_1$ (resp. $\cF_2$) passing through $z.$
Let $\gamma$ be any loop in $M$ starting at $z.$ We can approximate $\gamma$ by a zig-zag piecewise horizontal-vertical loop
(also starting at $z_0$) which is homotopic to $\gamma.$ Using the parallel transport to commute vertical pieces with horizontal pieces,
one sees easily that $\gamma$ is homotopic to the concatenation  $\gamma_1 +\gamma_2$ where $\gamma_1:[0,1/2] \to \cF_1(z)$ and
$\gamma_2: [1/2,1] \to \cF_2(z)$ are two paths lying in $\cF_1(z)$ and $\cF_2(z)$ respectively such that $\gamma_1(0) = \gamma_2(1)$
and the end point of $\gamma_1$ is the starting point of $\gamma_2.$ We will denote this point by 
${[\gamma]}. z  = \gamma_1(1/2) = \gamma_2(1/2).$ Observe that ${[\gamma]}. z \in \cF_1(z) \cap \cF_2(z),$ and it depends only on
the homotopy class $[\gamma]$ of $\gamma$ in the fundamental group $\pi_1(M,z).$

Fix a point $z_0 \in M.$ Denote by $\Gamma \in \pi_1(M,z_0)$ the set of elements $\alpha$ in the fundamental group $\pi_1(M,z_0)$
such that for any $z \in M$, any path $\mu$ from $z_0$ to $z$ we have $\pi_\mu(\alpha).z = z,$ where $\pi_\mu(\alpha)$ denotes the 
image of $\alpha$ in $\pi_1(M,z)$ via the natural isomorphism $\phi_\mu$ from $\pi_1(M,z_0)$ to $\pi_1(M,z)$ generated by the path
$\mu.$ One verifies easily that $\Gamma$ is a normal subgroup of $\pi_1(M,z_0),$ i.e. the quotient $G = \pi_1(M,z_0) / \Gamma$ is
a group. The group $\Gamma$ also satisfies the following remarkable property: if $[\gamma_1 + \gamma_2] \in \Gamma,$ where
$\gamma_1$ is a loop tangent to $\cF_1$ and $\gamma_2$ is a loop tangent to $\cF_2,$ then $[\gamma_1], [\gamma_2] \in \Gamma.$

Denote by $\tilde M$ the normal covering of $M$ associated to $\Gamma,$ i.e. $\pi_1(\tilde M) \cong \Gamma,$ $G = \pi_1(M,z_0)/\Gamma$
acts freely on $\tilde M,$ and $\tilde M / G = M.$ The foliations $\cF_1$ and $\cF_2$ can be lifted naturally to $\tilde M.$ We will denote
the lifted foliations on $\tilde M$ by $\tilde F_1$ and $\tilde F_2$ respectively. Since $(M,\cF_1,\cF_2)$ satisfies the completeness
condition, $(\tilde M, \tilde \cF_1, \tilde \cF_2)$ also satifies this condition. We want to show that $(\tilde M, \tilde \cF_1, \tilde \cF_2)$ has
direct product type. Due to the completeness condition, it is enough to verify that if $F_1$ is a leaf of $\tilde \cF_1$ and
$F_2$ is a leaf of $\tilde \cF_2$, then $F_1$ intesects with $F_2$ at exactly one point.

Assume to the contrary that $y_0,y_1 \in F_1 \cap F_2,$ $y_0 \neq y_1.$ Then there is a loop $\gamma = \gamma_1 + \gamma_2$ such that
$\gamma_1$ (resp. $\gamma_2$) starts at $y_0$ (resp. $y_1$), ends at $y_1$ (resp. $y_0$) and is tangent to $\tilde F_1$ (resp. $\tilde F_2$).
Denote the projection map from $\tilde M$ to $M$ by the hat, e.g. $\hat y_0 \in M$ is the image of $y_0$, $\hat \gamma_1$ is the image of $\gamma_1$
by the projection $\tilde M \to M.$ By construction, $[\hat \gamma] = [\hat \gamma_1 + \hat \gamma_2]
\in \Gamma,$ which implies that $\hat y_1 = \hat y_0,$ which in turns implies that $[\hat \gamma_1] \in \Gamma.$ But since $\gamma_1$ is a lifting
of $\gamma,$ the fact that $\hat \gamma_1 \in \Gamma$ implies that $y_1 = y_0$ by construction of $\tilde M,$ which is a contradiction. Thus any leaf
of $\tilde \cF_1$ intersects with any leaf of $\tilde \cF_2$ at exactly one point, and $(\tilde M, \tilde \cF_1, \tilde \cF_2)$ has
direct product type. The rest of the proof is straightforward.
\end{proof}

In Proposition \ref{prop:AlmostDirectProduct}, a-priori there are no Nambu structures. But of course, if $\cF_1$ and $\cF_2$ are generated
by two commuting Nambu structures, then we will have volume forms on the manifolds $F_1,F_2$ in the almost direct product model, and
the action of $G$ on $F_1$ and $F_2$ will be volume-preserving. Proposition \ref{prop:AlmostDirectProduct} can be naturally extended
to the case of $k$ foliations $\cF_1,\hdots, \cF_k$, where $k > 2$. But in the case $k > 2,$ the transversality condition 
$T_xM = \bigoplus_{i=1}^k T_x\cF_i$ is far from being sufficient for the decomposition of the picture in to a semi-direct product
 $(F_1 \times \hdots \times F_k)/G,$
and we really need some more meaningful commutativity condition. Fortunately, the commutativity of $k$ 
regular Nambu structures $\Lambda_1, \hdots, \Lambda_k$ which generate $\cF_1,\hdots, \cF_k$ will do the job.

\begin{prop} \label{prop:AlmostDirectProduct2}
Let $\cF_1,\hdots, \cF_2$ be $k$ foliations on a connected compact manifold $M$, generated by $k$ regular pairwise 
commuting Nambu structures $\Lambda_1,\hdots,\Lambda_k$ respectively, 
such that $T_xM = \bigoplus_{i=1}^k T_x\cF_i$ for any 
$x \in M.$ Then the multi-foliation $(M,\cF_1,\hdots, \cF_k)$ is isomorphic to an almost direct product model
\begin{equation}
(F_1 \times \hdots \times F_k) / G,
\end{equation}
where $F_1, \hdots, F_2$ are connected manifolds with (contravariant) volume forms, $G$ is a discrete group which acts  
on the product $F_1 \times F_2$ freely and diagonally, and the action of $G$ on each $F_i$ is volume-preserving. 
\end{prop}

The proof of Proposition \ref{prop:AlmostDirectProduct2} is absolutely similar to Proposition \ref{prop:AlmostDirectProduct},
and it can also be deduced from the proof of  Proposition \ref{prop:AlmostDirectProduct} by 
induction on $k$. (Notice that, for example, $\Lambda_1 \wedge \Lambda_2$ is again a regular Numbu structure which commutes
with the other $\Lambda_i$).

Proposition \ref{prop:AlmostDirectProduct2} is reminiscent of other almost direct product theorems, in particular the clasical theorem
about the almost direct product decomposition of reductive algebraic groups (see e.g. \cite{Milne-AlgebraicGroups2012}), and
also the topological decomposition theorem for nondegenerate singularities of integrable Hamiltonian systems \cite{Zung-Integrable1996}.
A particular case of the above theorem is when the foliations $\cF_i$ are one-dimensional, i.e. the Nambu tensors $\Lambda_1$
are commuting vector fields. In this case the manifold $M$ is a $k$-dimensional 
torus, $G$ is (in the generic case) isomorphic to $\bbZ^k$, and one recovers the classical Liouville's theorem about quasi-periodicity of
motion of integrable systems \cite{Liouville-1855}. Inspired by this, one can extend the notion of integrability of dynamical systems
to the case of foliations as follows:

\begin{defn}
 A foliation $\cF$ of dimension $q$ on a manifold $M$ will be called {\bf integrable} if it can be represented by a Nambu structure
$\Lambda= \Lambda_1$ or order $q$, and there are Nambu structures $\Lambda_2, \hdots, \Lambda_s$ and functions $F_1,\hdots,F_r$
such that: $q_1 + \hdots + q_s + r = n,$ $F_i$ are first integrals of $\Lambda_j$,  the $\Lambda_i$ commute pairwise, and
$\Lambda_1 \wedge \hdots \wedge \Lambda_s \neq 0$ almost everywhere.
\end{defn}

Proposition \ref{prop:AlmostDirectProduct2}, or rather a parametrized version of it which involves also first integrals, can then
be viewed as a generalization of the classical Lioville's theorem to the case of integrable foliations.

\section{Some final remarks and questions}

In this note, we considered only the regular case. But what about the singular case, when, for example, two foliations are regular but together
they have have singularities, or at least one of the two foliations is singular ? One should be able to develop a normal form theory
for such singular commuting foliations, at least in the case when the singularities are nondegenerate or generic in some sense.

What about differential forms which are invariants ? Apparently, those forms must be analogous to basic differential forms of
fibrations. (In particular, the contraction of the form with any vector field tangent to the foliation must vanish, i.e. they are
transverse forms, or more generally, one can talk about invariant transverse structures). Will they play a role in a generalized
theory of integrability of foliations . And what about a Galoisian theory of obstructions to the integrability of singular foliations ?

Since Nambu structures associated to foliations are uniquely defined only upto multiplication by a function, one may be tempted to weaken the commutativity condition and replace it by the following conformal commutativity condition: Two Nambu structures $\Lambda_1$ and $\Lambda_2$ of order $q_1$ and $q_2$ respectively, with $q_1+q_2\leq n$, are called \emph{conformally commutative} if locally near each point they become commutative after multiplication by local invertible functions. However, this notation of conformal commutativity does not lead to interesting geometric properties enjoyed by the true commutativity. For example, if $X_1$ and $X_2$ are two commuting vector fields on a torus $\Pi^2$ such that $[X_1,X_2]\neq 0$ everywhere, they are quasi-periodic, but the conformal commutativity condition is a trivial condition in this case and does not imply anything at all. So Nambu structure may be viewed as ``foliation structures'', and they may be even more interesting than foliations themselves.    

A very important use of the associated Nambu structures, which we didn't discuss in this paper, is that  they allow one to define the  \emph{deformation cohomology} of singular foliations and develop a deformation theory of singular foliations by algebraic means. This will be done in a forthcoming joint work of the authors and Philippe Monnier.


\vspace{0.5cm}


\begin{thebibliography}{10}
\baselineskip0.4cm

\bibitem{AG-Webs2000}
M.A. Akivis, V.V. Goldberg, {\it Differential geometry of webs}, 
Chapter I in Handbook of Differential Geometry, vol. 1 (2000), pp. 1-152, Elsevier. 

\bibitem{DufourZung-PoissonBook}
J.P. Dufour, N.T. Zung, Poisson structures and their normal forms, Progress in Mathematics, Vol. 242 (2005), Birkh\"auser.

\bibitem{EK-NonCommuting2005}
M. Einsiedler, A. Katok, {\it Rigidity of measures—the high entropy case and non-commuting foliations},
Israel J. Math. 148 (2005), 169–238.


\bibitem{Liouville-1855}
J. Liouville, {\it Note sur l’intégration des équations differentielles de la dynamique}, présentée
au bureau des longitudes le 29 juin 1853, 
Journal de Mathématiques pures et appliquées 20 (1855), 137-138.

\bibitem{Martinez-TopDegree2004}
D. Martinez Torres, {\it Global classification of generic vector fields of top degree}, J.
London Math. Soc., 69 (2004), 751--766.

\bibitem{Milne-AlgebraicGroups2012}
J.S. Milne, Algebraic groups , Lie groups, and their arithmetic subgroups, 2010 (available online).

\bibitem{Movshev-YM2009}
M.V. Movshev, {\it Yang-Mills theory and a superquadric}, 
Algebra, arithmetic, and geometry: in honor of Yu. I. Manin. 
Vol. II,  355–382, Progr. Math., 270 (2009), Birkhäuser Boston.

\bibitem{Takhtajan-Nambu1994}
L. Takhtajan, {\it On foundation of the generalized Nambu mechanics},
Comm. Math. Phys., 160 (1994), 295-315.

\bibitem{Walczak-Foliations}
P. Walczak, Dynamics of Foliations, Groups and Pseudogroups, Monografie Matematyczne, 
Vol. 64 (2004), Birkh\"auser, Basel.


\bibitem{Zung-Integrable1996}
N.T. Zung, {\it Symplectic topology of integrable Hamiltonian systems. 
I. Arnold-Liouville with singularities}, Compositio Math. 101 (1996), no. 2, 179-215.

\bibitem{Zung-Entropy2011}
N.T. Zung, {\it Entropy of geometric structures}, Bulletin Brazilian Math Soc, Volume 42 (2011), 
Number 4, 853--867.

\bibitem{Zung-Nambu2012}
N.T. Zung, {\it New results on the linearization of Nambu structures}, J. Math. Pures Appl, vol 99 (2013) No.2, 211-218. 


\end{thebibliography}
\end{document}